\documentclass[12pt, reqno]{amsart}
\usepackage{amsmath, amsthm, amscd, amsfonts, amssymb, graphicx, color}
\usepackage[bookmarksnumbered, colorlinks, plainpages]{hyperref}
\input{mathrsfs.sty}

\newtheorem{theorem}{Theorem}[section]
\newtheorem{lemma}[theorem]{Lemma}

\newtheorem{corollary}[theorem]{Corollary}
\theoremstyle{definition}

\theoremstyle{remark}

\numberwithin{equation}{section}

\begin{document}

\title[Matrix Hermite--Hadamard type inequalities]{Matrix Hermite--Hadamard type inequalities}

\author[M.S. Moslehian]{Mohammad Sal Moslehian}

\address{Department of Pure Mathematics, Center of Excellence in Analysis on Algebraic Structures (CEAAS), Ferdowsi University of
Mashhad, P.O. Box 1159, Mashhad 91775, Iran.}
\email{moslehian@ferdowsi.um.ac.ir and moslehian@member.ams.org}
\urladdr{\url{http://profsite.um.ac.ir/~moslehian/}}



\begin{abstract}
We present several matrix and operator inequalities of
Hermite--Hadamard type. We first establish a majorization version
for monotone convex functions on matrices. We then utilize the
Mond--Pecaric method to get an operator version for convex
functions. We also present some applications. Finally we obtain an
Hermite--Hadamard inequality for operator convex functions, positive
linear maps and operators acting on Hilbert spaces.
\end{abstract} \maketitle

\section{Introduction}

The following fundamental inequality, which was first published by
Hermite in 1883 in an elementary journal and independently proved in
1893 by Hadamard in \cite{JH}, is well known as the
Hermite--Hadamard inequality in the literature:
\begin{eqnarray}\label{s}
(y-x)f\left(\frac{x+y}{2}\right) \leq \int_0^1f(t)\,dt \leq
(y-x)\frac{f(x)+f(y)}{2}\,,
\end{eqnarray}
where $f$ is a convex function on an interval $[x,y]$. It provides a
two-sided estimate of the mean value of a convex function. If $f$ is
convex on a segment $[a,b]$ of a linear space, one can easily
observe that \eqref{s} is equivalent to the following double
inequality:
\begin{eqnarray}\label{m}
f\left(\frac{a+b}{2}\right) \leq \int_0^1f(ta+(1-t)b)\,dt \leq
\frac{f(a)+f(b)}{2}\,.
\end{eqnarray}
The Hermite--Hadamard inequality has several applications in
nonlinear analysis and the geometry of Banach spaces, see \cite{K}.
During the last decades several interesting generalizations, special
cases and formulations of this significant inequality for some types
of functions $f$ and various frameworks have been obtained. It gives
indeed a necessary and sufficient condition for a function $f$ to be
convex. We would like to refer the reader to \cite{12, 4, 2, 11, 3,
1, 8, 7, 6, 10} and references therein for more information. In
particular, Dragomir \cite{DRA1} very recently established an
operator version of the inequality for the operator convex
functions. In fact, in matrix analysis, there is an active area,
where some interesting matrix or norm inequalities are derived from
their scalar counterparts. Such inequalities may hold for operators
acting on an infinite dimensional separable Hilbert space. This is
based on the fact that self-adjoint operators (Hermitian matrices)
can be regarded as a generalization of real numbers. A natural
generalization of the classical Hermite--Hadamard inequality to
Hermitian matrices could be the double inequality
\begin{eqnarray}\label{false}
f\left(\frac{A+B}{2}\right) \leq \int_0^1f(tA+(1-t)B)\,dt \leq
\frac{f(A)+f(B)}{2}\,,
\end{eqnarray}
which is however not true, in general. To see this let us consider
the convex function $f(t)=t^3$ and matrices
$A={\scriptsize\left( \begin{array}{cc} 2 & 1 \\ 1 & 1 \\
\end{array}\right)}, B={\scriptsize\left( \begin{array}{cc} 1 & 0 \\
0 & 0 \\ \end{array}\right)}$. Then some straightforward
computations show that
$\left(\frac{A+B}{2}\right)^3={\scriptsize\left(
\begin{array}{cc} 17/4 & 7/4 \\ 7/4 & 3/4 \\ \end{array}\right)},\,\,
\int_0^1(tA+(1-t)B)^3\,dt={\scriptsize\left( \begin{array}{cc} 31/6
& 5/2
\\ 5/2 & 4/3 \\ \end{array}\right)},\,\,
\frac{A^3+B^3}{2}={\scriptsize\left( \begin{array}{cc} 7 & 4 \\
4 & 5/2 \\ \end{array}\right)}$ and that not both inequalities of
\eqref{false} simultaneously are true.

In this paper, we present some operator inequalities of
Hermite--Hadamard type in which we use the convexity instead of the
operator convexity. To do this, we first restrict ourselves to the
monotone convex functions to get a majorization version as our main
result. We then utilize the Mond--Pe\v{c}ari\'c method \cite{mp1,
seo, pms} to get another operator version of inequality \eqref{m}.
We also present some applications. Finally we generalize the main
result of \cite{DRA1} for operator convex functions, positive linear
maps and operators on (not necessarily finite dimensional) Hilbert
space.

\section{Preliminaries}

Let $\mathbb{B}(\mathscr{H})$ denote the algebra of all bounded
linear operators acting on a complex Hilbert space
$(\mathscr{H},\langle \cdot,\cdot\rangle)$ and $I_{\mathscr{H}}$ is
the identity operator. In the case where $\dim \mathscr{H} =n$, we
identify $\mathbb{B}(\mathscr{H})$ with the full matrix algebra
$\mathcal{M}_n$ of all $n \times n$ matrices with entries in the
complex field $\mathbb{C}$. We denote by $\mathcal{H}_{n}(J)$ the
set of all Hermitian matrices in $\mathcal{M}_{n}$, whose spectra
are contained in an interval $J \subseteq \mathbb{R}$. By $I_{n}$ we
denote the identity matrix of $\mathcal{M}_{n}$. An operator $ A\in
\mathbb{B}(\mathscr{H})$ is called positive (positive-semidefinite
for matrices) if $\langle A\xi, \xi\rangle \geq 0$ holds for every
$\xi\in \mathscr{H}$ and then we write $A\geq 0$. In particular, if
$A$ is invertible and positive (positive-definite for matrices),
then we write $A>0$. For self-adjoint operators $A,B \in
\mathbb{B}(\mathscr{H})$, we say $A\leq B$ if $B-A\geq0$. A map
$\Phi$ between $C^*$-algebras of operators is called positive if
$\Phi(A)\geq 0$ whenever $A\geq 0$. Throughout the paper all
real-valued functions are assumed to be continuous. A real-valued
function $f$ defined on an interval $J$ is called operator convex if
$f(\lambda A + (1-\lambda)B) \leq \lambda f(A)+(1-\lambda)f(B)$ for
all self-adjoint operators $A, B \in \mathbb{B}(\mathscr{H})$ with
spectra in $J$ and all $\lambda \in [0,1]$. Of course, there are
several equivalent version of the operator convexity in the
literature, see \cite[Chapter I]{seo} and \cite{mosl} and references
therein.

For a Hermitian matrix $A\in \mathcal{M}_n$, we denote by
$\lambda_{1}(A)\geq \lambda _{2}(A)\geq \cdots \geq \lambda _{n}(A)$
the eigenvalues of $A$ arranged in the decreasing order with their
multiplicities counted. The notation $\lambda(A)$ stands for the row
vector $(\lambda_{1}(A), \lambda _{2}(A), \cdots, \lambda _{n}(A))$.
The eigenvalue inequality $\lambda(A) \leq \lambda(B)$ means that
$\lambda_j(A)\leq \lambda_j(B)$ for all $1 \leq j \leq n$. As a
matter of fact for any two Hermitian matrices $A, B$ the inequality
$\lambda(A)\leq \lambda(B)$ holds if and only if $A\leq U^*BU$ for
some unitary matrix $U$. The weak majorization $\lambda(A) \prec_w
\lambda(B)$ means $\sum_{j=1}^k\lambda_j(A)\leq
\sum_{j=1}^k\lambda_j(B)\,\,(k=1, 2, \ldots, n)$. It is known that
three kinds of orders defined above satisfy $A\leq B \Rightarrow
\lambda(A) \leq \lambda (B) \Rightarrow \lambda(A) \prec_w
\lambda(B)$. A norm $\left\vert \left\vert \left\vert \cdot
\right\vert \right\vert \right\vert $ on $\mathcal{M}_{n}$ is said
to be unitarily invariant if $\left\vert \left\vert \left\vert
UAV\right\vert \right\vert \right\vert =\left\vert \left\vert
\left\vert A\right\vert \right\vert \right\vert $ for all $A\in
\mathcal{M}_{n}$ and all unitary matrices $U,V\in \mathcal{M}_{n}$.
The Ky Fan norms, the Schatten $p$-norms and the operator norm
provide significant families of unitarily invariant norms. The Ky
Fan dominance theorem states that $\lambda(A) \prec_w \lambda(B)$ if
and only if $\left\vert \left\vert \left\vert A\right\vert
\right\vert\right\vert \leq \left\vert \left\vert \left\vert
B\right\vert \right\vert \right\vert $ for all unitarily invariant
norms $\left\vert \left\vert \left\vert \cdot \right\vert
\right\vert\right\vert $. For more information on matrix analysis
the reader is referred to \cite{BHA1}.

\section{Operator Hermite--Hadamard type inequalities for convex functions}

We start this section by recalling two useful lemmas.

\begin{lemma}\cite[Lemma 2.4 and Remark 2.5]{SMS} (see also \cite[p. 281]{BHA1} and \cite[Theorem 2.3]{A-S})
\label{lem1}  Let $A\in \mathcal{H}_{n}(J)$, $f$ be a convex
function defined on $J$, $x\in \mathbb{C}^{m}$ and $\Phi:
\mathcal{M}_{n} \to \mathcal{M}_{m}$ be a positive linear map. If
either (i) $\Phi$ is unital and $\|x\|=1$ or (ii) $\|x\| \leq 1,
0\in J, f(0)\leq 0$ and $0<\Phi (I_{n})\leq I_{m}$, then
\begin{equation*}
f(\langle \Phi (A)x,x\rangle )\leq \langle \Phi (f(A))x,x\rangle\,.
\end{equation*}
\end{lemma}

\begin{lemma}
\cite[p. 67]{BHA1}\label{lem2} If $A \in \mathcal{H}_n$, then
\begin{equation*}
\sum_{j=1}^k \lambda_j (A)= \max \sum_{j=1}^k \langle A x_j,x_j
\rangle\qquad (1\le k \le n),
\end{equation*}
where the maximum is taken over all choices of orthonormal vectors $
x_1,x_2,\cdots,x_k \in \mathbb{C}^n$\,.
\end{lemma}

We are ready to give the operator version of the first inequality of
the Hermite--Hadamard inequality.
\begin{theorem}\label{t1}
Let $A, B\in \mathcal{H}_{n}(J)$, $f$ be a convex function on $J$
and $\Phi$ be a positive linear map from $\mathcal{M}_{n}$ to
$\mathcal{M}_{m}$. If either (i) $\Phi$ is unital or (ii) $0\in J,
f(0)\leq 0$ and $0<\Phi (I_{n})\leq I_{m}$, then
\begin{eqnarray*}
\lambda\left(f\left(\frac{\Phi(A)+\Phi(B)}{2}\right)\right) \prec_w
\lambda\left(\Phi\left(\int_0^1f(tA+(1-t)B)\,dt\right)\right)\,.
\end{eqnarray*}
\end{theorem}

\begin{proof}
Suppose that $\lambda _{1}, \cdots , \lambda _{m}$ are the
eigenvalues of $\frac{\Phi(A)+\Phi(B)}{2}$ with $u_{1}, \cdots
,u_{m}$ as an orthonormal system of corresponding eigenvectors
arranged such that $f(\lambda _{1})\geq f(\lambda _{2})\geq \cdots
\geq f(\lambda _{m})$. We have
\begin{align*}
\sum_{j=1}^k&\lambda_j\left(f\left(\frac{\Phi(A)+\Phi(B)}{2}\right)\right)\\
&= \sum_{j=1}^kf\left(\left\langle \frac{\Phi(A)+\Phi(B)}{2}u_j,u_j\right\rangle\right)\qquad\qquad\qquad({\rm by~our~assumption~on~} u_j)\\
&\leq \sum_{j=1}^k\int_0^1f\left(\left\langle t\Phi(A)+(1-t)\Phi(B)u_j,u_j\right\rangle\right)\,dt\\
&\qquad\qquad\qquad\qquad\qquad\quad({\rm by~the~classical~Hermite--Hadamard~inequality~})\\
&= \sum_{j=1}^k\int_0^1f\left(\left\langle \Phi(tA+(1-t)B)u_j,u_j\right\rangle\right)\,dt\qquad\quad\quad\quad({\rm by~the~linearity~of~} \Phi)\\
&\leq \sum_{j=1}^k\int_0^1\left\langle \Phi\big(f(tA+(1-t)B)\big)u_j,u_j\right\rangle \,dt\quad\qquad\qquad\qquad({\rm by~Lemma~} \ref{lem1})\\
&= \sum_{j=1}^k \left\langle \int_0^1\Phi\big(f(tA+(1-t)B)\big)\,dt\,u_j,u_j\right\rangle\\
&\qquad\qquad\qquad\quad\qquad({\rm by~the~linearity~and~continuity~of~the~inner~product})\\
&\leq \sum_{j=1}^k \lambda_j\left( \int_0^1\Phi\big(f(tA+(1-t)B)\big)\,dt\right)\quad\qquad\qquad\qquad\quad({\rm by~Lemma~} \ref{lem2})\\
&= \sum_{j=1}^k\lambda_j\left(\Phi\left(\int_0^1f(tA+(1-t)B)dt\right)\right)\\
&\qquad\qquad\qquad\qquad\qquad\qquad\qquad\quad\quad({\rm
by~the~linearity~and~continuity~of~} \Phi).
\end{align*}
\end{proof}

Using Theorem \ref{t1} with $\Phi(A)=A$ we obtain that
\begin{corollary}
If $A, B\in \mathcal{H}_{n}([\omega,\Omega])$ and $f$ is a convex
function on $[\omega,\Omega]$, then
\begin{eqnarray*}
\lambda\left(f\left(\frac{A+B}{2}\right)\right)\prec_w
\lambda\left(\int_0^1f(tA+(1-t)B)\,dt\right) \,.
\end{eqnarray*}
In particular,
\begin{eqnarray*}
{\it Tr}\left(f\left(\frac{A+B}{2}\right)\right)\leq {\it
Tr}\left(\int_0^1f(tA+(1-t)B)\,dt\right)\,.
\end{eqnarray*}
\end{corollary}

The fact that the function $f(t)=t^r$ is convex for $r>1$ yields
that
\begin{corollary}
Let $r>1$, $A, B\in \mathcal{H}_{n}([\omega,\Omega])$ and $\Phi:
\mathcal{M}_{n}\to \mathcal{M}_{m}$ be a positive linear map such
that either (i) it is unital or (ii) $0\in J, f(0)\leq 0$ and
$\Phi(I_n)\leq I_m$. Then
\begin{eqnarray*}
\left\vert \left\vert
\left\vert\left(\frac{\Phi(A)+\Phi(B)}{2}~\right)^r\right\vert
\right\vert \right\vert \leq \left\vert \left\vert
\left\vert\int_0^1\Phi((tA+(1-t)B)^r)\,dt~\right\vert \right\vert
\right\vert\,.
\end{eqnarray*}
\end{corollary}

Now we get some operator versions of the second inequality of the
Hermite--Hadamard inequality \label{hh} in two fashions. The first
version is for monotone convex functions and the second version,
which is weaker than the first one, is just for convex functions. To
present the first version we would extend the following interesting
result of Bourin to the positive linear maps.

\begin{lemma}
\cite[Theorem 2.2]{BOU}\label{l3} Let $A_1, \cdots, A_k \in
\mathcal{H}_{n}([\omega,\Omega])$ and $f$ be an increasing convex
function defined on $[\omega,\Omega]$ containing the spectra of
$A_i,\,i=1, \cdots, k$. If $Z_1, \cdots, Z_k$ are matrices with
$\sum_{i=1}^kZ_i^*Z_i=I_n$, then there is a unitary matrix $U$ such
that $f\left(\sum_{i=1}^kZ_i^*A_iZ_i\right)\leq
U\left(\sum_{i=1}^kZ_i^*f(A_i)Z_i\right)U^*$.
\end{lemma}
\begin{theorem}\label{t2}
Let $A_1, \cdots, A_k \in \mathcal{H}_{n}([\omega,\Omega])$ and $f$
be an increasing convex function defined on $[\omega,\Omega]$
containing the spectra of $A_i,\,i=1, \cdots, k$. If $\Phi_1,
\cdots, \Phi_k: \mathcal{M}_{n} \to \mathcal{M}_{m}$ are positive
linear maps such either (i)  $\sum_{i=1}^k\Phi_i(I_n)= I_m$ or (ii)
$0 \in J, f(0) \leq 0$ and $\sum_{i=1}^k\Phi_i(I_n)\leq I_m$, then
there is a unitary matrix $U$ such that
$f\left(\sum_{i=1}^k\Phi_i(A_i)\right)\leq
U\sum_{i=1}^k\Phi_i(f(A_i))U^*$.
\end{theorem}
\begin{proof}
First let us prove Lemma \ref{l3} whenever $0 \in J, f(0) \leq 0$ and $\sum_{i=1}^k\Phi_i(I_n)\leq I_m$:\\
Lemma \ref{l3} with $k=1$ and $0 \in J, f(0)\leq 0, Z^*Z \leq I_n$
instead of $Z^*Z=I_n$ is still true. In fact, due to $I_n-Z^*Z  \geq
0$, there is a  matrix $Y$ such that $Z^*Z+Y^*Y=I_n$. Using Lemma
\ref{l3}, we have
$$f(Z^*AZ)=f(Z^*AZ+Y^*0Y) \leq U(Z^*f(A)Z+Y^*f(0)Y)U^*\leq U^*Z^*f(A)ZU^*$$
for some unitary $U$. The general case now follows by considering
$Z$ to be the column vector $(Z_1, \cdots, Z_k)$ and $A$ to be the
diagonal matrix $A={\rm diag}(A_1, \cdots, A_k)$.

Second assume that $A$ is a Hermitian matrix and $\Psi:
\mathcal{M}_{n} \to \mathcal{M}_{m}$ is a positive linear map. Using
the spectral decomposition $A=\sum_j\lambda_jE_j$ of $A$, the fact
that
$\sum_j\sqrt{\Psi(E_j)}\sqrt{\Psi(E_j)}=\sum_j\Psi(E_j)=\Psi(I_n)$,
Lemma \ref{l3} and the paragraph above we have
\begin{align}\label{msm}
f(\Psi(A))&=f\left(\sum_j\lambda_j\Psi(E_j)\right)=f\left(\sum_j\sqrt{\Psi(E_j)}\lambda_j\sqrt{\Psi(E_j)}\right)\nonumber\\
&\leq U\left(\sum_j\sqrt{\Psi(E_j)}f(\lambda_j)\sqrt{\Psi(E_j)}\right)U^*=U\left(\sum_jf(\lambda_j)\Psi(E_j)\right)U^*\nonumber\\
&=U\Psi\left(\sum_jf(\lambda_j)E_j\right)U^*=U\Psi(f(A))U^*
\end{align}
for some unitary $U$.\\
Next assume that $A_1, \cdots, A_k \in
\mathcal{H}_{n}([\omega,\Omega])$. Set $$\Psi({\rm diag}(A_1,
\cdots, A_k))=\sum_{i=1}^k\Phi_i(A_i)\,.$$ Then $\Psi$ is clearly a
positive linear map. Hence there is a unitary $U$ such that
\begin{align*}
f\left(\sum_{i=1}^k\Phi_i(A_i)\right)&=f(\Psi({\rm diag}(A_1, \cdots, A_k)))\\
&\leq U\Psi(f({\rm diag}(A_1, \cdots, A_k)))U^*\quad\qquad\qquad\qquad\qquad\quad\,\,({\rm by~} \eqref{msm})\\
&= U\Psi({\rm diag}(f(A_1), \cdots, f(A_k)))U^*\quad({\rm by~the~functional~calculus})\\
&=U\sum_{i=1}^k\Phi_i(f(A_i))U^*\,.
\end{align*}
\end{proof}

We are in a situation to give a matrix version of the second
inequality of the Hermite--Hadamard inequality.

\begin{theorem} \label{t3} Let $A, B\in \mathcal{H}_{n}([\omega,\Omega])$, $f$ be an increasing convex function
on $[\omega,\Omega]$ and $\Phi: \mathcal{M}_n \to \mathcal{M}_m$ be
a positive linear map such that either (i) it is unital or (ii)
$0\in J, f(0)\leq 0$ and $\Phi(I_n)\leq I_m$. If there is a unitary
$U$ such that $f(t\Phi(A)+(1-t)\Phi(B))\leq
U\left[t\Phi(f(A))+(1-t)\Phi(f(B))\right]U^*$ for all $t \in [0,
1]$, then
\begin{align}\label{main1}
\lambda\left(\int_0^1f(\Phi(tA+(1-t)B))\right)\leq
\lambda\left(\frac{\Phi(f(A))+\Phi(f(B))}{2}\right)\,.
\end{align}
\end{theorem}
\begin{proof}
By the assumption,
\begin{eqnarray*}
f(t\Phi(A)+(1-t)\Phi(B))\leq
U\left[t\Phi(f(A))+(1-t)\Phi(f(B))\right]U^*
\end{eqnarray*}
for some unitary $U$ and all $t \in [0,1]$. Hence
\begin{eqnarray*}
\int_0^1f(\Phi(tA+(1-t)B))&=&\int_0^1f(t\Phi(A)+(1-t)\Phi(B))\,dt\\
&\leq &\int_0^1 U\left[t\Phi(f(A))+(1-t)\Phi(f(B))\right]U^*\,dt\\
&= & U\,\int_0^1t\Phi(f(A))+(1-t)\Phi(f(B))\,dt\,U^*\\
&= & U\left[\frac{\Phi(f(A))+\Phi(f(B))}{2}\right]U^*\,.
\end{eqnarray*}
Thus we get \eqref{main1}.
\end{proof}

Now we use the Mond--Pe\v{c}ari\'{c} method \cite{seo} to get the
second version of the second inequality of the Hermite--Hadamard
inequality \label{hh}.

\begin{theorem} \label{t4} Let $A, B\in \mathbb{B}(\mathscr{H})$ be self-adjoint operators with spectra in $[\omega,\Omega]$, $f: [\omega,\Omega] \to (0,\infty)$ be a convex function and $\Phi: \mathcal{M}_{n} \to \mathcal{M}_{m}$ be a positive linear map. Then
\begin{align}\label{main2}
\Phi&\left(\int_0^1f(tA+(1-t)B)\,dt\right) \nonumber\\
&\leq
\max\left\{\frac{\Omega-t}{\Omega-\omega}\cdot\frac{f(\omega)}{f(t)}+\frac{t-\omega}{\Omega-\omega}\cdot\frac{f(\Omega)}{f(t)}:
t \in[\omega,\Omega]\right\}\frac{f(\Phi(A))+f(\Phi(B))}{2}\,.
\end{align}
\end{theorem}
\begin{proof}
Let $A, B$ be Hermitian operators with the spectra in
$[\omega,\Omega]$. It follows from the convexity of $f$ that
\begin{eqnarray*}
f(t)=f\left(\frac{\Omega-t}{\Omega-\omega}.\omega+\frac{t-\omega}{\Omega-\omega}.\Omega\right)\leq
\frac{\Omega-t}{\Omega-\omega}f(\omega)+\frac{t-\omega}{\Omega-\omega}f(\Omega)
\end{eqnarray*}
for all $t\in[\omega,\Omega]$. Applying the functional calculus we
obtain
$$
f(tA+(1-t)B) \leq
\frac{\Omega-tA+(1-t)B}{\Omega-\omega}f(\omega)+\frac{tA+(1-t)B-\omega}{\Omega-\omega}f(\Omega)\,.$$
So that
\begin{eqnarray*}\Phi\left(f(tA+(1-t)B)\right) &\leq& \frac{\Omega-t\Phi(A)+(1-t)\Phi(B)}{\Omega-\omega}f(\omega)\\
&&+\frac{t\Phi(A)+(1-t)\Phi(B)-\omega}{\Omega-\omega}f(\Omega)\,,
\end{eqnarray*}
whence for each unit vector $x\in \mathscr{H}$, we get
\begin{align*}
\langle \Phi\left(f(tA+(1-t)B)\right)x,x\rangle &\leq \frac{\Omega-\langle(t\Phi(A)+(1-t)\Phi(B))x,x\rangle}{\Omega-\omega}f(\omega)\\
&+\frac{\langle(t\Phi(A)+(1-t)\Phi(B))x,x\rangle-\omega}{\Omega-\omega}f(\Omega).
\end{align*}
So
\begin{align*}
\int_0^1\langle \Phi\left(f(tA+(1-t)B)\right)x,x\rangle \,dt &\leq  \frac{\Omega-\int_0^1\langle(t\Phi(A)+(1-t)\Phi(B))x,x\rangle \,dt}{\Omega-\omega}f(\omega)\\
&+\frac{\int_0^1\langle(t\Phi(A)+(1-t)\Phi(B))x,x\rangle
\,dt-\omega}{\Omega-\omega}f(\Omega).
\end{align*}
Hence
\begin{align*}
\left\langle \int_0^1\Phi\left(f(tA+(1-t)B)\right)\,dt\,x,x\right\rangle &\leq \frac{\Omega-\left\langle\frac{\Phi(A)+\Phi(B)}{2}x,x\right\rangle}{\Omega-\omega}f(\omega)\\
&+\frac{\left
\langle\frac{\Phi(A)+\Phi(B)}{2}x,x\right\rangle-\omega}{\Omega-\omega}f(\Omega).
\end{align*}
Therefore
\begin{align*}
\left\langle \Phi\left(\int_0^1f(tA+(1-t)B)\,dt\right)\,x,x\right\rangle &\leq  \frac{\Omega-\left\langle\frac{\Phi(A)+\Phi(B)}{2}x,x\right\rangle}{\Omega-\omega}f(\omega)\\
&+\frac{\left
\langle\frac{\Phi(A)+\Phi(B)}{2}x,x\right\rangle-\omega}{\Omega-\omega}f(\Omega).
\end{align*}
Hence
\begin{align*}
\frac{\langle \Phi\left(\int_0^1f(tA+(1-t)B)\,dt\right)\,x,x\rangle}{f\left(\left\langle\frac{\Phi(A)+\Phi(B)}{2}x,x\right\rangle\right)} &\leq \frac{1}{f\left(\left\langle\frac{\Phi(A)+\Phi(B)}{2}x,x\right\rangle\right)}\\
& \hspace{-1.1in}\times \left(\frac{\Omega-\left\langle\frac{\Phi(A)+\Phi(B)}{2}x,x\right\rangle}{\Omega-\omega}f(\omega)+\frac{\left \langle\frac{\Phi(A)+\Phi(B)}{2}x,x\right\rangle-\omega}{\Omega-\omega}f(\Omega)\right).\\
\end{align*}
Thus
\begin{align*}
\Big\langle \Phi\Big(\int_0^1f(tA+&(1-t)B)\,dt\Big)x,x\Big\rangle\\
&\leq \alpha f\left(\left\langle\frac{\Phi(A)+\Phi(B)}{2}x,x\right\rangle\right)\\
&\leq \alpha f\left(\frac{\langle\Phi(A)x,x\rangle+\langle\Phi(B)x,x\rangle}{2}\right)\\
&\leq \alpha\, \frac{f(\langle\Phi(A)x,x\rangle)+f(\langle\Phi(B)x,x\rangle)}{2}\qquad\quad({\rm by~the~convexity~of~} f)\\
&\leq \alpha\, \frac{\langle f(\Phi(A))x,x\rangle+\langle f(\Phi(B))x,x\rangle}{2}\quad({\rm by~} f(\langle Ax,x\rangle \leq \langle f(A)x,x\rangle)\\
&\leq \left\langle \alpha\, \frac{f(\Phi(A))+ f(\Phi(B))}{2}x,x\right\rangle\,,\\
\end{align*}
where
$\alpha=\max\left\{\frac{\Omega-t}{\Omega-\omega}\cdot\frac{f(\omega)}{f(t)}+\frac{t-\omega}{\Omega-\omega}\cdot\frac{f(\Omega)}{f(t)}:
t \in[\omega,\Omega]\right\}$. Hence inequality \eqref{main2} holds.
\end{proof}

It follows from the Ky Fan Dominance Theorem (see \cite{KY}),
Theorem \ref{t1} and Theorem \ref{t4} that
\begin{corollary}
Let $A, B\in \mathcal{H}_{n}([\omega,\Omega])$, $f$ be a convex
function on $[\omega,\Omega]$ and $\Phi$ be a positive linear map
from $\mathcal{M}_{n}$ to $\mathcal{M}_{m}$. If either (i) $\Phi$ is
unital or (ii) $0\in [\omega,\Omega], f(0)=0$ and $0<\Phi
(I_{n})\leq I_{m}$, then
\begin{align*}
&\left\vert \left\vert \left\vert~f\left(\frac{\Phi(A)+\Phi(B)}{2}~\right)\right\vert \right\vert \right\vert \\
&\quad\leq \left\vert \left\vert \left\vert\Phi~\left(\int_0^1f(tA+(1-t)B)\,dt\right)~\right\vert \right\vert \right\vert\\
&\quad\leq
\max\left\{\frac{\Omega-t}{\Omega-\omega}\cdot\frac{f(\omega)}{f(t)}+\frac{t-\omega}{\Omega-\omega}\cdot\frac{f(\Omega)}{f(t)}:
t \in[\omega,\Omega]\right\}\left\vert \left\vert
\left\vert~\frac{f(\Phi(A))+f(\Phi(B))}{2}\right\vert \right\vert
\right\vert\,.
\end{align*}
\end{corollary}
The mapping $\Phi(A)=\sum_{i=1}^kX_i^*AX_i$ is a positive linear
map. So that we infer the following result from Theorem \ref{t1} and
Theorem \ref{t4}.
\begin{corollary}
Let $A, B\in \mathcal{H}_{n}([\omega,\Omega])$, $f$ be a convex
function on $[\omega,\Omega]$ and $X_1, \cdots, X_k\in
\mathcal{M}_{n}$ such that $\sum_{i=1}^kX_i^*X_i=I_n$. Then
\begin{eqnarray*}
&&\hspace{-1.2in}\left\vert \left\vert \left\vert~f\left(\frac{1}{2}\sum_{i=1}^kX_i^*(A+B)X_i~\right)\right\vert \right\vert \right\vert\\
&\leq& \left\vert \left\vert \left\vert\sum_{i=1}^kX_i^*\int_0^1f(tA+(1-t)B)\,dtX_i~\right\vert \right\vert \right\vert\\
&\leq& \max\left\{\frac{\Omega-t}{\Omega-\omega}\cdot\frac{f(\omega)}{f(t)}+\frac{t-\omega}{\Omega-\omega}\cdot\frac{f(\Omega)}{f(t)}: t \in[\omega,\Omega]\right\}\\
&& \times\left\vert \left\vert
\left\vert~\frac{f(\sum_{i=1}^kX_i^*AX_i)+f(\sum_{i=1}^kX_i^*BX_i)}{2}\right\vert
\right\vert \right\vert\,.
\end{eqnarray*}
\end{corollary}

\section{Operator Hermite--Hadamard type inequalities for operator convex functions}

In this section we generalize the main result of \cite{DRA1}.
\begin{theorem}
If $A, B$ are self-adjoint operators on a Hilbert space $H$ with
spectra in an interval $J$, $f$ is an operator convex function on
$J$ and $k, p$ are positive integers, then
\begin{align}\label{0}
&f\left(\frac{A+B}{2}\right) \leq \frac{1}{k^p}\sum_{i=0}^{k^p-1}f\left(\frac{2i+1}{2k^p} A + \left(1-\frac{2i+1}{2k^p}\right)B\right)\nonumber\\
&\leq \int_0^1f(tA+(1-t)B)\,dt\nonumber\\
&\leq \frac{1}{2k^p}\sum_{i=0}^{k^p-1}\left[f\left(\frac{i+1}{k^p} A + \left(1-\frac{i+1}{k^p}\right)B\right)+f\left(\frac{i}{k^p} A + \left(1-\frac{i}{k^p}\right)B\right)\right]\nonumber\\
&\leq \frac{f(A)+f(B)}{2}\,.
\end{align}
\end{theorem}
\begin{proof}
Let $x\in H$ be a unit vector. It is easy to see that the function
$\rho(t)=\langle f(tA+(1-t)B)x,x\rangle$ is a real-valued convex
function on the interval $[0,1]$, see \cite[Theorem 2.1]{DRA1}.
Utilizing the classical Hermite--Hadamard inequality on the interval
$[\frac{i}{k^p}, \frac{i+1}{k^p}]$, we get that
\begin{eqnarray*}
\rho\left(\frac{2i+1}{2k^p}\right) \leq
k^p\int_\frac{i}{k^p}^{\frac{i+1}{k^p}}\rho(t)\,dt \leq
\frac{\rho\left(\frac{i}{k^p}\right)+\rho\left(\frac{i+1}{k^p}\right)}{2}\,.
\end{eqnarray*}
Summation of the above inequalities over $i=0, 1, \cdots, k^p-1$
yields
\begin{eqnarray*}
\sum_{i=0}^{k^p-1}\rho\left(\frac{2i+1}{2k^p}\right) \leq
k^p\int_0^1\rho(t)\,dt \leq
\sum_{i=0}^{k^p-1}\frac{\rho\left(\frac{i}{k^p}\right)+\rho\left(\frac{i+1}{k^p}\right)}{2}\,.
\end{eqnarray*}
Hence
\begin{align}\label{1}
&\frac{1}{k^p}\sum_{i=0}^{k^p-1} f\left(\frac{2i+1}{2k^p} A + \left(1-\frac{2i+1}{2k^p}\right)B\right)\nonumber\\
&\leq \int_0^1f(tA+(1-t)B)\,dt\nonumber\\
&\leq \frac{1}{2k^p}\sum_{i=0}^{k^p-1}\left[f\left(\frac{i+1}{k^p} A
+ \left(1-\frac{i+1}{k^p}\right)B\right)+ f\left(\frac{i}{k^p} A +
\left(1-\frac{i}{k^p}\right)B\right)\right]\,.
\end{align}
By the operator convexity of $f$ we have
\begin{align}\label{2}
\frac{1}{k^p}\sum_{i=0}^{k^p-1}& f\left(\frac{2i+1}{2k^p} A + \left(1-\frac{2i+1}{2k^p}\right)B\right)\nonumber\\
&\geq f\left[\frac{1}{k^p}\sum_{i=0}^{k^p-1}\left(\frac{2i+1}{2k^p} A + \left(1-\frac{2i+1}{2k^p}\right)B\right)\right]\nonumber\\
&= f\left[\frac{\sum_{i=0}^{k^p-1}(2i+1)}{2k^{2p}}A + \left(1-\frac{\sum_{i=0}^{k^p-1}(2i+1)}{2k^{2p}}B\right)\right]\nonumber\\
&=f\left(\frac{A+B}{2}\right)
\end{align}
and
\begin{align}\label{3}
&\frac{1}{2k^p}\sum_{i=0}^{k^p-1}\left[f\left(\frac{i+1}{k^p} A + \left(1-\frac{i+1}{k^p}\right)B\right)+ f\left(\frac{i}{k^p} A + \left(1-\frac{i}{k^p}\right)B\right)\right]\nonumber\\
&\leq \frac{1}{2k^p}\sum_{i=0}^{k^p-1}\left[\frac{i+1}{k^p}f(A) + \left(1-\frac{i+1}{k^p}\right)f(B) + \frac{i}{k^p}f(A) + \left(1-\frac{i}{k^p}\right)f(B)\right]\nonumber\\
&=\frac{f(A)+f(B)}{2}\,.
\end{align}
Now \eqref{1}, \eqref{2} and \eqref{3} yield the whole inequalities
\eqref{0} as desired.
\end{proof}

\bibliographystyle{amsplain}

\end{document}